\documentclass[12pt,a4paper]{article}

\usepackage{amsmath,amsthm,amssymb,color}
\usepackage[authoryear]{natbib}
\usepackage{booktabs}

\usepackage{bbm}
\usepackage{mathrsfs}
\usepackage[breaklinks=true]{hyperref}
\usepackage{graphicx}
\usepackage{tikz}
\usetikzlibrary{arrows, automata}
\usetikzlibrary{calc}
\usetikzlibrary{positioning}

\numberwithin{equation}{section}
\allowdisplaybreaks[4]

\theoremstyle{plain}

\newtheorem{theorem}{Theorem}[section]
\newtheorem{proposition}{Proposition}[section]
\newtheorem{lemma}{Lemma}[section]

\theoremstyle{definition}

\newtheorem{remark}{Remark}[section]

\makeatletter

\newcount\minute
\newcount\hour
\newcount\hourMins
\def\now{%
\minute=\time%
\hour=\time \divide \hour by 60%
\hourMins=\hour \multiply\hourMins by 60%
\advance\minute by -\hourMins%
\zeroPadTwo{\the\hour}:\zeroPadTwo{\the\minute}%
}
\def\zeroPadTwo#1{\ifnum #1<10 0\fi#1}

\renewcommand{\cite}{\citet}

\def\^#1{\ifmmode {\mathaccent"705E #1} \else {\accent94 #1} \fi}
\def\~#1{\ifmmode {\mathaccent"707E #1} \else {\accent"7E #1} \fi}

\def\*#1{#1^\ast}
\edef\-#1{\noexpand\ifmmode {\noexpand\bar{#1}} \noexpand\else \-#1\noexpand\fi}
\def\>#1{\vec{#1}}
\def\.#1{\dot{#1}}

\def\atop{\@@atop}
\def\*#1{\mathscr{#1}}

\renewcommand{\leq}{\leqslant}
\renewcommand{\geq}{\geqslant}

\newcommand{\de}{\delta}

\newcommand{\eq}{\eqref}

\newcommand{\Var}{\mathop{\mathrm{Var}}\nolimits}

\def\be#1{\begin{equation*}#1\end{equation*}}
\def\ben#1{\begin{equation}#1\end{equation}}
\def\bes#1{\begin{equation*}\begin{split}#1\end{split}\end{equation*}}
\def\besn#1{\begin{equation}\begin{split}#1\end{split}\end{equation}}

\def\beqn#1\eeqn{\begin{align}#1\end{align}}
\def\beq#1\eeq{\begin{align*}#1\end{align*}}

\usepackage{graphicx}
\usepackage{latexsym}
\usepackage{amsmath,amsthm,amssymb,amscd}
\usepackage{epsf,amsmath}

\newcommand{\g}{\gamma}

\renewcommand\section{\@startsection {section}{1}{\z@}%
{-3.5ex \@plus -1ex \@minus -.2ex}%
{1.3ex \@plus.2ex}%
{\center\small\sc\mathversion{bold}\MakeUppercase}}

\def\subsection#1{\@startsection {subsection}{2}{0pt}%
{-3.5ex \@plus -1ex \@minus -.2ex}%
{1ex \@plus.2ex}%
{\bf\mathversion{bold}}{#1}}

\def\subsubsection#1{\@startsection{subsubsection}{3}{0pt}%
{\medskipamount}%
{-10pt}%
{\normalsize\itshape}{\kern-2.2ex. #1.}}

\def\blfootnote{\xdef\@thefnmark{}\@footnotetext}

\makeatother

\begin{document}

\title{Skewness correction in tail probability approximations for sums of local statistics
}

\author{Xiao Fang\footnote{Department of Statistics, The Chinese University of Hong Kong, Hong Kong. Email: xfang@sta.cuhk.edu.hk}, Li Luo\footnote{Department of Statistics, The Chinese University of Hong Kong, Hong Kong. Email: luoli@link.cuhk.edu.hk} and Qi-Man Shao\footnote{1. Department of Statistics and Data Science, The Southern University of Science and Technology, China. 2. Department of Statistics, The Chinese University of Hong Kong, Hong Kong. Email: qmshao@sta.cuhk.edu.hk}}

\date{}
\maketitle

\noindent{\bf Abstract:}
Correcting for skewness can result in more accurate tail probability approximations in the central limit theorem for sums of independent random variables. In this paper, we extend the theory to sums of local statistics of independent random variables and apply the result to $k$-runs, U-statistics, and subgraph counts in the Erd\"os-R\'enyi random graph.
To prove our main result, we develop exponential concentration inequalities and higher-order Cram\'er-type moderate deviations via Stein's method.

\medskip

\noindent{\bf AMS 2010 subject classification:} 60F05

\noindent{\bf Keywords and phrases:}
Stein's method, skewness correction, moderate deviations, local dependence, $k$-runs, U-statistics, Erd\"os-R\'enyi random graph

\section{Introduction}

Let $W_n=\sum_{i=1}^n X_i/\sqrt{n}$ where $\{X_1, X_2, \dots\}$ are independent and identically distributed (i.i.d.) with $E X_1=0, E X_1^2=1$, and $E e^{t_0 X_1}<\infty$ for a constant $t_0>0$.
It is known that (cf. \cite[Chapter VIII, Theorem 1]{Pe75})
\ben{\label{14}
\left| \frac{P(W_n> x)}{1-\Phi(x)}-1  \right|\leq  \frac{C(1+x^3)}{\sqrt{n}}\ \text{for}\ 0\leq x\leq C_0 n^{1/6}
}
and
\ben{\label{54}
\left| \frac{P(W_n> x)}{(1-\Phi(x))e^{\g x^3/6}}-1  \right|\leq C\big(\frac{1+x}{\sqrt{n}}+\frac{x^4}{n} \big)\ \text{for}\ 0\leq x\leq C_0 n^{1/4},
}
where $C_0$ is any fixed constant, $\Phi$ denotes the standard normal distribution function, $\g=E W_n^3=E X_1^3/\sqrt{n}$ and $C$ is a positive constant depending only on $t_0$ and $C_0$.
We refer to results such as \eq{14} and \eq{54} as Cram\'er-type moderate deviation results.
The range $0\leq x=o(n^{1/6})$ ($0\leq x=o(n^{1/4})$ resp.), for the relative error in \eq{14} (\eq{54} resp.) to vanish is optimal.
We refer to the modification of the normal distribution function in \eq{54} as skewness correction.

We are interested in extending the theory of skewness correction for tail probability approximations to sums of local statistics of independent random variables as follows.
For a positive integer $m$, let $\{X_1,\dots, X_m\}$ be a sequence of independent random variables.
Let  
$$W=\sum_{i=1}^n \xi_i, $$
where for each $i\in \{1,\dots, n\}$,  $\xi_i$ is a function of a small subset of $\{X_1,\dots, X_m\}$.
Absolute-error bounds in normal approximation for such $W$ are well studied in the literature.
See, for example, \cite{ChSh04} for results under a more general local dependence setting.
However, the accuracy of tail probability approximations for such $W$ is less well understood.
Recently, \cite{Zh19} considered Cram\'er-type moderate deviations as in \eq{14} for such $W$.
Our main result is a general relative-error bound (cf. \eq{13}) for $\Big| \frac{P(W>x)}{(1-\Phi(x))e^{\g x^3/6}} -1 \Big|$, where $\g=E W^3$,
under certain boundedness conditions (cf. \eq{11}).
For standardized sums of i.i.d., bounded random variables, our bound vanishes for the correct range $0\leq x=o(n^{1/4})$, although the rate is suboptimal (cf. \eq{46}).
We apply our main result to $k$-runs, U-statistics, and subgraph counts in the Erd\"os-R\'enyi random graph.
In each application, our bound vanishes for presumably the correct range of $x$ in terms of the system size.

We use Stein's method, which was introduced by \cite{St72} for normal approximation, to prove our main result.
\cite{ChGoSh11} provided an introduction to the method and a survey of its recent developments.
\cite{ChFaSh13} developed the method to prove Cram\'er-type moderate deviation results in normal approximation
for dependent random variables under a boundedness condition.
\cite{ChFaSh13b} and \cite{ShZhZh18} obtained Cram\'er-type moderate deviation results in Poisson approximation and non-normal approximations, respectively.
\cite{Zh19} refined the results in \cite{ChFaSh13} by relaxing the boundedness condition.
\cite[Chapter 4]{Br17} obtained a Cram\'er-type moderate deviation result in a higher-order approximation for the Erlang-C queuing model.
His proof relies heavily on explicit expressions of certain conditional expectations in the model.
To prove our general bound, we develop Stein's method for exponential concentration inequalities (cf. Proposition \ref{p2}) and for higher-order Cram\'er-type moderate deviations. For the latter, we use $P(Z_\g>x)$ in place of $(1-\Phi(x))e^{\g x^3/6}$ for an intermediate approximation, where $Z_\g$ follows a suitable standardized Poisson distribution.

Related results are available in the literature.
(a). Asymptotic expansions in the central limit theorem have been extensively studied.
See, for example, \cite{Pe75} for the classical Edgeworth expansion and \cite{Ba86} and \cite{RiRo03} for related expansions using Stein's method. These expansions require either a continuity condition on the random variable or a smoothness condition on certain test functions.
The $O(1/\sqrt{n})$ rate of convergence in the absolute-error bound for normal approximation for sums of $n$ independent discrete random variables generally can not be improved.
Nevertheless, \eq{54}, as well as our main result, shows that
it is still possible to improve the accuracy in terms of the relative error in tail probability approximations using an appropriate expansion.
(b). In the proof of our main result, we use a standardized Poisson distribution for an intermediate approximation.
Translated Poisson distributions have been proposed as alternatives to normal distributions to approximate lattice random variables in the total variation distance. See, for example, \cite{Ro05, Ro07}, \cite{BaLuXi18a, BaLuXi18b}, and \cite{BaXi18}.
Instead of matching the support of random variables as in these results, we use standardized Poisson distributions to correct for skewness.
See \cite{Ri09} for a similar use of standardized Poisson distributions.

The remainder of this paper is organized as follows.
In Section 2, we state the general relative-error bound in normal approximation with skewness correction for sums of local statistics of independent random variables and discuss applications to $k$-runs, U-statistics, and subgraph counts in the Erd\"os-R\'enyi random graph.
In Section 3, we prove an exponential concentration inequality, which is crucial to the proof of the general bound.
In Section 4, we prove the general bound.
\section{Main results}

\subsection{A general relative-error bound}

For a positive integer $N$, denote $[N]:=\{1,\dots, N\}$.
Let $m$ and $n$ be positive integers. Let $\{X_\alpha: \alpha\in [m]\}$ be a sequence of independent random variables.
Let $W=\sum_{i=1}^n \xi_i$, where each $\xi_i$ is a function of $\{X_\alpha: \alpha\in \mathcal{I}_i\}$ for some $\mathcal{I}_i\subset [m]$.
For $\alpha\in [m]$, let $N_\alpha=\{i\in [n]: \alpha\in \mathcal{I}_i\}$.

\begin{theorem}\label{t1}
Under the above setting, assume that $E \xi_i=0$ for each $i\in [n]$ and $\Var(W)=1$. Assume further that
\ben{\label{11}
|\xi_i|\leq \delta, \quad |\mathcal{I}_i|\leq s,\quad |N_\alpha|\leq d,
}
where $|\cdot|$ denotes the cardinality when applied to a set.
Denote $\g:=E W^3$.
Let $C_0$ be any fixed constant.
For
\ben{\label{25}
0\leq x\leq C_0  (m n s^4 d^4 \delta^5 ) ^{-1/2},
}
we have
\besn{\label{13}
\Big| \frac{P(W>x)}{(1-\Phi(x))e^{\g x^3/6}} -1 \Big|
\leq
C m n s^4 d^4 \delta^5(1+x^2),
}
where $C$ is a positive constant depending only on $C_0$.
\end{theorem}

Clearly, applying the above result to $-W$ yields  
$$
\Big| \frac{P(W<-x)}{\Phi(-x)e^{-\g x^3/6}}-1   \Big| 
\leq
C m n s^4 d^4 \delta^5(1+x^2).
$$

To illustrate that the range of $x$ for the relative error in our approximation to vanish is correct, we first consider the standardized sums of i.i.d., bounded random variables.
Let $X_1, X_2, \dots$ be i.i.d.\ with $E X_i=0, E X_i^2=1, |X_i|\leq C_1<\infty$.
For an integer $n\geq 1$,
let $\xi_i=X_i/\sqrt{n}$ for each $i\in [n]$ and let $W=\sum_{i=1}^n \xi_i$.
This satisfies the assumptions in Theorem \ref{t1} with
\be{
m=n,\quad \delta=\frac{C_1}{\sqrt{n}},\quad s=1, \quad d=1, \quad \g=\frac{E X_1^3}{\sqrt{n}}.
}
Let $C_0$ be any fixed constant.
From \eq{13}, we have, for $0\leq x\leq C_0 n^{1/4}$,
\ben{\label{46}
\Big| \frac{P(W>x)}{(1-\Phi(x))e^{\g x^3/6}} -1 \Big|
\leq \frac{C}{\sqrt{n}}(1+x^2),
}
where $C$ is a positive constant depending only on $C_0$ and $C_1$.
Note that according to \eq{54},
the range $x=o(n^{1/4})$ for the relative-error bound in \eq{46} to vanish is optimal.
However, due to the suboptimality of our concentration inequality (cf. Remark \ref{r1}), our rate of convergence in \eq{46} is not optimal.

\begin{remark}\label{r1}
Corrections to the normal distribution function can be formally generalized by accounting for the 4th and higher cumulants.
However, one obstacle to obtaining a complete proof for even higher-order expansions is that our exponential concentration inequality (cf. Proposition \ref{p2}) is only useful in the range $x=o(n^{1/4})$.
\end{remark}

\subsection{Applications}

\subsubsection{$k$-runs}
Let $n>k>1$ be integers. Let $p\in (0,1)$.
Let $X_1,\dots, X_n$ be i.i.d.\ and $P(X_i=1)=1-P(X_i=0)=p$.
Let
\be{
W=\sum_{i=1}^n \xi_i,\  \xi_i=\frac{X_i X_{i+1}\cdots X_{i+k-1}-p^k}{\sigma},
}
where $\sigma$ is the normalizing constant such that $\Var(W)=1$, and $X_{n+i}:=X_i$ for $i\geq 1$.
It satisfies the assumptions in Theorem \ref{t1} with
\be{
m=n,\quad \de=\frac{1}{\sigma},\quad s=k, \quad d=k.
}
Therefore, we obtain:
\begin{proposition}\label{p3}
Let $\g=E W^3$ with the $W$ above. 
Let $C_0$ be any fixed constant.
We have, for $0\leq x\leq C_0 (\sigma^5/n^2 k^8)^{1/2}$,
\be
{
\max\left\{\Big| \frac{P(W<-x)}{\Phi(-x)e^{-\g x^3/6}}-1   \Big| ,\Big| \frac{P(W>x)}{(1-\Phi(x))e^{\g x^3/6}} -1 \Big|\right\}
\leq \frac{Cn^2 k^8}{\sigma^5}(1+x^2),
}
where $C$ is a positive constant depending only on $C_0$.
\end{proposition}

In Propositions \ref{p3} and \ref{p4}, the formulation of the problem is not symmetric; therefore, we state the bound for both the left and right tail probabilities. The computation of $\sigma^2$ and $\gamma$ is not central to our study and is omitted from this and the next two examples.
If $k$ and $p$ are fixed, then the range of $x$ for the relative-error bound to vanish is $0\leq x=o(n^{1/4})$, which is presumably optimal in comparison to the i.i.d.\ case.

In the following, we provide empirical evidence of the advantage of skewness correction. Consider $k=2$. It can be computed that
\be{
\sigma^2=n(p^2+2p^3-3p^4)
}
and
\be{
\g=\frac{n}{\sigma^3}(p^2+6p^3-3p^4-24p^5+20p^6).
}
In the following table, we provide simulated values (based on $10^6$ repetitions) for
\be{
L_N:=\frac{P(W<-x)}{\Phi(-x)}-1,\quad L_{skew}:=\frac{P(W<-x)}{\Phi(-x)e^{-\g x^3/6}}-1,
}
and
\be{
R_N:=\frac{P(W>x)}{1-\Phi(x)}-1,\quad R_{skew}:=\frac{P(W>x)}{(1-\Phi(x))e^{\g x^3/6}}-1,
}
for the case $n=1500$ and $p=0.25$ and various values of $x$.
The table clearly shows that the tail probability approximations with skewness correction is much more accurate.

\begin{table}[h]
\caption{
$n=1500, p=0.25, \g\approx 0.138$.
Values of $L_N(L_{skew})$ and $R_N(R_{skew})$
based on $10^6$ repetitions.
}
\begin{center}
\begin{tabular}{|c|c|c|c|c|}
\hline
$x$ &$L_N$  & $L_{skew}$  & $R_N$  & $R_{skew}$  \\\hline
2 & -0.195  & -0.032 & 0.262    & 0.050   \\
2.5 & -0.238   & 0.093 & 0.344  & -0.063   \\
3 & -0.538  & -0.138 & 0.476   & -0.208   \\
3.5 & -0.811   & -0.491 & 1.201  & -0.182   \\
4 & -0.968   & -0.862 & 1.810   & -0.358   \\ \hline
        \end{tabular}

        \label{table1}
\end{center}
\end{table}

\subsubsection{U-statistics}

Let $X_1, X_2,\dots$ be a sequence of i.i.d.\ random variables from a fixed distribution.
Let $s\geq 2$ be a fixed integer.
Let $h: \mathbb{R}^s\to \mathbb{R}$ be a fixed, symmetric, Borel-measurable function.
We consider the \cite{Ho48} U-statistic
\be{
\sum_{1\leq i_1<\dots <i_s\leq m} h(X_{i_1},\dots, X_{i_s}).
}
Assume that
\be{
E h(X_1,\dots, X_s)=0,\   |h(X_1,\dots, X_s)|\leq C_1<\infty.
}
and the U-statistic is non-degenerate, namely,
\be{
E g^2(X_1)>0,
}
where
\be{
g(x):=E(h(X_1,\dots, X_s)|X_1=x).
}
Applying Theorem \ref{t1} to the U-statistic above yields the following result:
\begin{proposition}
In the above setting, let
\be{
W=\frac{1}{\sigma} \sum_{1\leq i_1<\dots <i_s\leq m} h(X_{i_1},\dots, X_{i_s}),
}
where
\be{
\sigma^2=\Var \big[  \sum_{1\leq i_1<\dots <i_s\leq m} h(X_{i_1},\dots, X_{i_s})   \big].
}
Let $\g=E W^3$.
Let $C_0$ be any fixed constant.
We have, for $0\leq x\leq C_0 m^{1/4}$,
\be{
\Big| \frac{P(W>x)}{(1-\Phi(x))e^{\g x^3/6}} -1 \Big|\\
\leq \frac{C}{\sqrt{m}}(1+x^2)
}
where $C$ is a positive constant depending only on $C_0$ and $h$.
\end{proposition}

\begin{proof}
The above $W$ satisfies the assumptions in Theorem \ref{t1} with
\be{
n={m \choose s},\quad \de=\frac{C_1}{\sigma}, \quad d\leq m^{s-1}.
}
By the non-degeneracy condition. $\sigma^2\asymp m^{2s-1}$.
The proposition then follows from \eq{13}.
\end{proof}

\begin{remark}
\cite{ChSh07} obtained a bound on the Kolmogorov distance in normal approximation for non-degenerate U-statistics.
The references therein comprise a large body of literature on the rate of convergence in normal approximation for U-statistics.
Our relative error bound for the skewness corrected tail probability approximation for U-statistics seems to be new.
\end{remark}

\subsubsection{Subgraph counts in the Erd\"os-R\'enyi random graph}

Let $K(N,p)$ be the Erd\H{o}s-R\'enyi random graph with $N$ vertices.
Each pair of vertices is connected with probability $p$ and remains disconnected with probability $1-p$, independent of all else.
Let $G$ be a given fixed graph.
For any graph $H$, let $v(H)$ and $e(H)$ denote the number of its vertices and edges, respectively.
Let $v=v(G), e=e(G)$.
Theorem \ref{t1} leads to the following result.
\begin{proposition}\label{p4}
Let $S$ be the number of copies (not necessarily induced) of $G$ in $K(N,p)$, and let $W=(S-E S)/\sqrt{\Var(S)}$ be the standardized version.
Let $\g=E W^3$.
Let $C_0$ be any fixed constant.
We have, for $0\leq x\leq C_0 [N^6 (1-p)^{5/2} p^{5e}/\psi^{5/2}]^{1/2}$,
\be{
\max\left\{\Big| \frac{P(W<-x)}{\Phi(-x)e^{-\g x^3/6}}-1   \Big| ,\Big| \frac{P(W>x)}{(1-\Phi(x))e^{\g x^3/6}} -1 \Big|\right\}\\
\leq \frac{C(G)\psi^{5/2}}{(1-p)^{5/2} p^{5e}N^6}(1+x^2),
}
where $C(G)$ is a constant depending only on $C_0$ and $G$, and
\be{
\psi=\min_{H\subset G, e(H)>0} \{N^{v(H)}p^{e(H)}\}.
}
\end{proposition}

\begin{proof}
In this proof, $C$ denotes positive constants that are allowed to depend on $C_0$ and the given fixed graph $G$.
Let the potential edges of $K(N,p)$ be denoted by $(e_1,\dots, e_{{N\choose 2}})$.
In applying Theorem \ref{t1}, let $W=\sum_{i\in I} X_i$,
where the index set is
\be{
I=\Big\{ i=(i_1,\dots, i_e): 1\leq i_1<\dots <i_e\leq {N\choose 2}, G_i:=(e_{i_1}, \dots, e_{i_e}) \ \text{is a copy of $G$}    \Big\},
}
\be{
X_i=\sigma^{-1} \big(  Y_i  -p^e\big),\quad \sigma^2:=\Var(S),\quad Y_i=\Pi_{l=1}^e E_{i_l},
}
and $E_{i_l}$ is the indicator of the event that the edge $e_{i_l}$ is connected in $K(N,p)$.
The above $W$ satisfies the assumptions in Theorem \ref{t1} with
\be{
n:=|I|\leq N^v,\quad m={N \choose 2},\quad \de=\frac{1}{\sigma},\quad s\leq C,\quad d\leq C N^{v-2}.
}
It is known that (cf. (3.7) of \cite{BaKaRu89})
\be{
\sigma^2\geq C (1-p) N^{2v} p^{2e} \psi^{-1}.
}
The proposition then follows from \eq{13}.
\end{proof}

\begin{remark}
\cite{BaKaRu89} first studied normal approximation for the above $W$ using Stein's method.
Because $\psi\leq N^2 p$, if $p$ is fixed, then the range of $x$ for the relative error to vanish is $o(N^{1/2})$.
It is larger than the range of $o(N^{1/3})$, for which \cite{Zh19} proved that the relative error in normal approximation vanishes.
\end{remark}

\section{Exponential concentration inequality}

\subsection{Preliminaries}

Let $\{X_\alpha': \alpha\in [m]\}$ be an independent copy of $\{X_\alpha: \alpha\in [m]\}$.
For each $\alpha\in [m]$, let $W^{\{\alpha\}}$ be defined as for $W$ at the beginning of Section 2.1, except by changing $X_\alpha$ to $X_\alpha'$.
We have
\ben{\label{6}
\mathcal{L}(W,W^{\{\alpha\}})=\mathcal{L}(W^{\{\alpha\}},W).
}
From \eq{11}, we have
\ben{\label{49}
|W-W^{\{\alpha\}}|\leq 2d\de.
}
By the Efron-Stein inequality, we have
\ben{\label{5}
C_2:=\sum_{\alpha=1}^m E(W-W^{\{\alpha\}})^2\geq 2\Var(W)=2.
}
Moreover, it is straightforward to verify that $1\leq n\de$ and
\ben{\label{47}
1=\Var(W)\leq n sd \de^2,\ \sum_{\alpha=1}^m E(W-W^{\{\alpha\}})^2\leq 4m d^2 \de^2,\ |\g|\leq 4ns^2d^2\de^3.
}

We have the following local dependence structure for $W$.
\underline{(LD1)}: For $i\in [n]$, let $A_i=\{j\in [n]: \mathcal{I}_j\cap \mathcal{I}_i\ne \emptyset\}$; hence, $\xi_i$ is independent of $\{\xi_j: j\notin A_i\}$. \underline{(LD2)}: For $i\in [n]$ and $j\in A_i$, let $A_{ij}=\{k\in [n]: \mathcal{I}_k\cap (\mathcal{I}_i\cup \mathcal{I}_j)\ne \emptyset\}$; hence, $\{\xi_i, \xi_j\}$ is independent of $\{\xi_k: k\notin A_{ij}\}$. \underline{(LD3)}: For $i\in [n]$, $j\in A_i$ and $k\in A_{ij}$, let $A_{ijk}=\{l\in [n]: \mathcal{I}_l\cap (\mathcal{I}_i\cup \mathcal{I}_j\cup \mathcal{I}_k)\ne \emptyset\}$; hence, $\{\xi_i, \xi_j, \xi_k\}$ is independent of $\{\xi_l: l\notin A_{ijk}\}$.
From \eq{11}, we have
\ben{\label{48}
|A_i|, |A_{ij}|, |A_{ijk}|\leq 3sd.
}
For $A\subset [n]$, denote $\xi_A=\sum_{i\in A} \xi_i$ and $\xi_{i}:=\xi_{\{i\}}$.
We have
\ben{\label{50}
\gamma=E W^3=2\sum_{i=1}^n \sum_{j\in A_i} \sum_{k\in A_{ij}} E\xi_i\xi_{j}\xi_{k}-\sum_{i=1}^n \sum_{j\in A_i} E\xi_i \xi_{j}^2.
}
Let
\be{
V_1=\sum_{\alpha=1}^m (W-W^{\{\alpha\}}),\quad V_2=\sum_{\alpha=1}^m (W-W^{\{\alpha\}})^2.
}
\begin{lemma}\label{l10}
Regard $V_1$ and $V_2$ as functions of the independent random variables $\{X_\alpha: \alpha\in [m]\}\cup \{X_\alpha': \alpha\in [m]\}$.
For some $\beta\in [m]$, if we change $X_\beta$ or $X_\beta'$ to another independent copy, $V_1$ is changed by at most $2 s d \de$, and $V_2$ is changed by at most $4 s d^2 \de^2$.
\end{lemma}

\begin{proof}[Proof of Lemma \ref{l10}]
For each $\alpha\in [m]$, define $\{\xi_i^{\{\alpha\}}: i\in [n]\}$ as for $\{\xi_i: i\in [n]\}$, except by changing $X_\alpha$ to $X_\alpha'$. We have
\be{
W^{\{\alpha\}}=\sum_{i=1}^n \xi_i^{\{\alpha\}}.
}
From the definition of $N_\alpha$ and $\mathcal{I}_i$, we have
\be{
V_1=\sum_{\alpha=1}^m \sum_{i\in [n]} (\xi_i-\xi_i^{\{\alpha\}})=\sum_{\alpha=1}^m \sum_{i\in N_\alpha} (\xi_i-\xi_i^{\{\alpha\}})=\sum_{i=1}^n \sum_{\alpha\in \mathcal{I}_i} (\xi_i-\xi_i^{\{\alpha\}} ).
}
Changing $X_\beta$ or $X_\beta'$ only affects $(\xi_i-\xi_i^{\{\alpha\}} )$ if $i\in N_\beta$, which has cardinality at most $d$ by \eq{11}. From $|\mathcal{I}_i|\leq s$ and $|\xi_i-\xi_i^{\{\alpha\}}|\leq 2\de$ (cf. \eq{11}), $V_1$ is changed by at most $2sd\de$.

Now we turn to $V_2$. We have
\bes{
V_2=&\sum_{\alpha=1}^m \big[\sum_{i\in N_\alpha} (\xi_i-\xi_i^{\{\alpha\}}) \big]^2\\
=& \sum_{\alpha=1}^m \sum_{i, j \in N_\alpha}(\xi_i-\xi_i^{\{\alpha\}}) (\xi_j-\xi_j^{\{\alpha\}})\\
=&\sum_{i, j \in N_\alpha \atop \mathcal{I}_i\cap \mathcal{I}_j \ne \emptyset}
\sum_{\alpha\in \mathcal{I}_i\cap \mathcal{I}_j} (\xi_i-\xi_i^{\{\alpha\}}) (\xi_j-\xi_j^{\{\alpha\}}).
}
Reasoning similar to that for $V_1$ above leads to the observation that changing $X_\beta$ or $X_\beta'$ changes $V_2$ by at most $4sd^2 \de^2$.
\end{proof}

\subsection{Moment generating function bound}

\begin{proposition}\label{p1}
Let $C_0$ be any fixed constant.
Under the assumptions in Theorem \ref{t1},
for
\be{
0\leq t \leq C_0  (n s^2 d^2 \de^3 ) ^{-1/2},
}
we have
\ben{\label{4}
E e^{tW}\leq C\exp(\frac{t^2}{2}+\frac{\gamma t^3}{6}),
}
where $C$ is a positive constant depending only on $C_0$.
\end{proposition}

\begin{proof}
In this proof, $C$ denotes positive constants that can depend on $C_0$, $O(a)$ denotes a quantity such that $|O(a)|\leq C a$.
Let $h(t)=E e^{tW}$.
Note that from $t=O(1)(ns^2d^2\de^3)^{-1/2}$ and $n\de\geq 1$, we have
\ben{\label{53}
sd\de t=O(1).
}
Because $\xi_i$ is independent of $W-\xi_{A_i}$ by (LD1), $E \xi_i=0$, $|A_i|\leq Csd$ from \eq{48} and $|\xi_i|\leq \delta$ from \eq{11}, we have
\besn{\label{1}
h'(t)=& E W e^{tW}=\sum_{i=1}^n E \xi_i e^{tW}=\sum_{i=1}^n E \xi_i [e^{tW}-e^{t(W-\xi_{A_i})}]\\
=&\sum_{i=1}^n E \xi_i\big[ \xi_{A_i} t e^{tW} -\frac{\xi_{A_i}^2}{2} t^2 e^{tW} +O(s^3d^3 \delta^3 t^3 e^{tW+Csd \delta t})     \big].
}
For the first term on the right-hand side of \eq{1}, we have,
recalling $\sum_{i=1}^n E \xi_i \xi_{A_i}=E W^2=1$ and using similar arguments as above for the error term,
\besn{\label{2}
&\sum_{i=1}^n E \xi_i \xi_{A_i} t e^{tW}\\
=& \sum_{i=1}^n \sum_{j\in A_i}E \xi_i \xi_{j} t E e^{t(W-\xi_{A_{ij}})}+\sum_{i=1}^n \sum_{j\in A_i} E \xi_i \xi_{j} t [e^{tW}-e^{t(W-\xi_{A_{ij}})}]\\
=& th(t)+\sum_{i=1}^n \sum_{j\in A_i} E \xi_i \xi_{j} t E [e^{t(W-\xi_{A_{ij}})}-e^{tW}]
+\sum_{i=1}^n \sum_{j\in A_i} E \xi_i \xi_{j} t [e^{tW}-e^{t(W-\xi_{A_{ij}})}]\\
=& t h(t) +\sum_{i=1}^n \sum_{j\in A_i} E [\xi_i \xi_{j}-E \xi_i \xi_{j}] \xi_{A_{ij}} t^2 e^{tW}+O(ns^3d^3 \delta^4 t^3 e^{tW+Csd\delta t}).
}
For the second terms on the right-hand of \eq{1} and of \eq{2}, we have, by recalling \eq{50},
\besn{\label{3}
&\sum_{i=1}^n E[\xi_i \xi_{A_i} \xi_{A_{ij}}-(E \xi_i \xi_{A_i}) \xi_{A_{ij}}-\xi_i \xi_{A_i}^2/2] t^2 e^{tW}\\
=&\sum_{i=1}^n \sum_{j\in A_i}\sum_{k\in A_{ij}}  E[\xi_i \xi_{j} \xi_{k}-(E \xi_i \xi_{j}) \xi_{k}-\frac{\xi_i \xi_{j} \xi_k I(k=j)}{2}] t^2 E e^{t(W-\xi_{A_{ijk}})}\\
&+ \sum_{i=1}^n \sum_{j\in A_i}\sum_{k\in A_{ij}}  E[\xi_i \xi_{j} \xi_{k}-(E \xi_i \xi_{j}) \xi_{k}-\frac{\xi_i \xi_{j} \xi_k I(k=j)}{2}] t^2  [e^{tW}-e^{t(W-\xi_{A_{ijk}})}]\\
=&\gamma \frac{t^2 h(t)}{2}+O(ns^3d^3\delta^4 t^3 e^{tW+Csd\delta t}).
}
Combining \eq{1}, \eq{2} and \eq{3}, we have
\ben{\label{22}
h'(t)=th(t)+\gamma \frac{t^2 h(t)}{2} +O(ns^3 d^3\delta^4 t^3  e^{Csd \delta t})h(t).
}
Recall \eq{53}. Because $h(0)=1$, we have
\bes{
\log h(t)=&\int_0^t [u+\gamma u^2/2+O(ns^3d^3\delta^4 u^3)] du\\
=&\frac{t^2}{2}+\frac{\gamma t^3}{6}+O(ns^3d^3\delta^4 t^4).
}
This implies \eq{4} because from $t = O(1)  (n s^2 d^2 \de^3 ) ^{-1/2}$ and \eq{47}, we have
\ben{\label{52}
ns^3d^3\delta^4 t^4\leq \frac{Cns^3d^3\delta^4}{n^2s^4d^4\de^6}=\frac{C}{nsd\de^2}\leq C.
}
\end{proof}

\subsection{Exponential concentration inequality}

What we call a concentration inequality here is a smoothing inequality originally used in normal approximation by \cite{Es45}. It was developed via Stein's method in, for example, \cite{HoCh78} and \cite{ChSh04}. \cite{Sh10} developed exponential concentration inequalities in normal approximation for non-linear statistics.

\begin{proposition}\label{p2}
Let $C_0$ be any fixed constant.
Under the assumptions of Theorem \ref{t1},
for $d\de\leq 1/2$ and
\be{
1\leq x \leq C_0  (n s^2 d^2 \de^3 ) ^{-1/2},
}
we have, for any $\varepsilon>0$,
\be{
P(x\leq W\leq x+\varepsilon)\leq Cm s^2 d^2 \de^2 (\varepsilon+d \de) e^{\varepsilon x} x \exp(-\frac{x^2}{2}+\frac{\gamma x^3}{6}) + \exp(-\frac{1}{Cms^2d^4\de^4}),
}
where $C$ is a positive constant depending only on $C_0$.
\end{proposition}

To prove Proposition \ref{p2}, we apply the following lemma, which provides moment generating function bounds for a function of independent random variables. It is proved in a manner similar to that in \cite{Ch07}. See \cite{Ch08} and \cite{ChRo10} for related ideas.

\begin{lemma}\label{l7}
Let $V=h(Y_1,\dots, Y_N)$ where $(Y_1,\dots, Y_N)$ are independent.
Assume that $EV=0$.
Let $(\tilde{Y}_1,\dots, \tilde{Y}_N)$ be an independent copy of $(Y_1,\dots, Y_N)$.
Suppose that for any $i\in [N]$,
\be{
|h(Y_1,\dots, Y_N)-h(Y_1,\dots, Y_{i-1}, \tilde{Y}_i, Y_{i+1},\dots, Y_N)|\leq \de_3.
}
Then we have, for any $\theta>0$,
\be{
E e^{\theta V}\leq \exp(N \de_3^2 \theta^2/4).
}
\end{lemma}

\begin{proof}[Proof of Lemma \ref{l7}]
Let $V_0=V$ and for $i\in [N]$, let
\be{
V_i=h(\tilde{Y_1},\dots, \tilde{Y}_{i}, Y_{i+1},\dots, Y_N)
}
and
\be{
U_i'=h(Y_1,\dots, Y_{i-1}, \tilde{Y}_{i}, Y_{i+1},\dots, Y_N).
}
For $a>0$ and $\theta\geq 0$, let $m_a(\theta)=E e^{\theta (V\wedge a)}$.
Because $V_N$ is independent of $V$ and $E V_N=0$, we have
\be{
m_a'(\theta)=E (V\wedge a)e^{\theta (V\wedge a)}\leq E V e^{\theta (V\wedge a)}=E(V-V_N) e^{\theta (V\wedge a)}=\sum_{i=1}^N E(V_{i-1}-V_i) e^{\theta (V\wedge a)}.
}
Note that $E(V_{i-1}-V_i)e^{\theta (V\wedge a)}=E(V_i-V_{i-1}) e^{\theta (U_i'\wedge a)}$, which is a consequence of the exchangeability of $Y_i$ and $\tilde{Y}_i$.
Therefore,
\be{
m_a'(\theta)\leq \frac{1}{2}\sum_{i=1}^N  E(V_{i-1}-V_i) (e^{\theta (V\wedge a)}-e^{\theta (U_i'\wedge a)}).
}
From the fact that (cf. (7) of \cite{Ch07}) for any $x, y\in \mathbb{R}$,
\be{
\big|\frac{e^{x}-e^y}{x-y}\big|\leq \frac{1}{2}(e^x+e^y),
}
we have
\bes{
m_a'(\theta)\leq & \frac{\theta}{4} \sum_{i=1}^N E|V_{i-1}-V_i||V-U_i'| (e^{\theta (V\wedge a)}+e^{\theta (U_i'\wedge a)})\\
=& \frac{\theta}{2} \sum_{i=1}^N E|V_{i-1}-V_i||V-U_i'| e^{\theta (V\wedge a)},
}
again by the exchangeability of $Y_i$ and $\tilde{Y}_i$.
From the boundedness conditions on $|V_{i-1}-V_i|$ and $|V-U_i'|$, we have
\be{
m_a'(\theta)\leq \frac{\theta}{2} N \de_3^2 m_a(\theta), \ \forall\ \theta\geq 0,
}
which implies
\be{
m_a(\theta)\leq \exp(N\de_3^2 \theta^2/4).
}
The lemma is proved by letting $a\to \infty$.

\end{proof}

\begin{proof}[Proof of Proposition \ref{p2}]
In this proof, we use $c$ and $C$ to denote positive constants that can depend only on $C_0$.
Recall $W^{\{\alpha\}}$ from Section 3.1.
Let $I$ be a uniform random variable on $[m]$ and independent of all else.
Let $W'=W^{\{I\}}$.
Similar to \cite{Sh10}, define
\be{
f(w)=
\begin{cases}
0, & w\leq x-2d\de\\
e^{xw}(w-x+2d\de), & x-2d\de< w\leq x+\varepsilon +2d\de\\
e^{x(x+\varepsilon+2d\de)} (\varepsilon+4d\de), & w>x+\varepsilon+2d\de.
\end{cases}
}
From \eq{6}, we have
\be{
\mathcal{L}(W, W')=\mathcal{L}(W',W);
}
Hence
\be{
E(W-W')(f(W)+f(W'))=0.
}
Rewrite it as
\be{
\text{LHS}:= 2E(W-W')f(W)=E(W-W')(f(W)-f(W')) =: \text{RHS}.
}

\noindent\emph{Part I: Upper bound for LHS.}

Averaging over $I$:
\be{
\text{LHS}=\frac{2}{m}\sum_{\alpha=1}^m E(W-W^{\{\alpha\}})f(W).
}
Recall
$V_1=\sum_{\alpha=1}^m (W-W^{\{\alpha\}})$ and note that $x-2d\de\geq 0$ by the assumptions of the proposition.
From the upper bound on $f$, we have
\besn{\label{10}
|\text{LHS}|&\leq \frac{2}{m} E|V_1| (\varepsilon+4d\de)e^{x(x+\varepsilon+2d\de)} I(W\geq x-2d\de)\\
&\leq \frac{2}{m} (\varepsilon+4d\de)e^{x(x+\varepsilon+2d\de)} \big[ E|V_1|I(|V_1|>M(x-2d\de)) +ME W I(W\geq x-2d\de)     \big],
}
where $M\geq 1$ is to be chosen above \eq{8}.
Note that $V_1$ is symmetrical.
For the first term on the right-hand side of \eq{10}, we have
\besn{\label{23}
&E|V_1|I(|V_1|>M(x-2d\de)) =2 E V_1I(V_1>M(x-2d\de))\\
\leq & 2 M (x-2d\de) P(V_1>M(x-2d\de))+2\int_{M(x-2d\de)}^\infty P(V_1>y) dy.
}
Applying Lemma \ref{l7} with $\theta=x$ and Lemma \ref{l10} to $V_1$, we have
\be{
E e^{xV_1}\leq \exp(Cm s^2 d^2 \de^2 x^2).
}
Therefore,
\bes{
&E|V_1|I(|V_1|>M(x-2d\de))\\
\leq & 2M(x-2d\de)\frac{e^{Cm s^2 d^2 \de^2 x^2}}{e^{xM(x-2d\de)}}+2\int_{M(x-2d\de)}^\infty \frac{e^{Cm s^2 d^2 \de^2 x^2}}{e^{xy}} dy \\
\leq & C e^{2Md\de x}e^{-M x^2} M x e^{Cm s^2 d^2 \de^2 x^2}.
}
Now we consider the second term on the right-hand side of \eq{10}.
Note that $|\g x|\leq C n s^2 d^2 \de^3 x\leq C$ (cf. \eq{47}) for the range of $x$ in the proposition to be non-empty.
Following reasoning similar to that for \eq{52} and \eq{53}, we have $ns^3 d^3 \de^4 x^3\leq ns^3 d^3 \de^4 x^4\leq C$ and 
\ben{\label{60}
sd\de x\leq C.
}
From the proof of Proposition \ref{p1} (cf. \eq{22}), we have
\be{
E W e^{xW}\leq Cx \exp(\frac{x^2}{2}+\frac{\gamma x^3}{6}).
}
Therefore, from \eq{60},
\be{
E W I(W\geq x-2d\de)\leq E We^{xW}/e^{x(x-2d\de)}\leq Cx \exp(-\frac{x^2}{2}+\frac{\gamma x^3}{6})
}
Combining the above bounds, we have
\ben{\label{51}
|\text{LHS}|\leq \frac{C}{m} e^{x^2 }(\varepsilon+d\de) e^{\varepsilon x} M x [e^{2Md\de x} e^{-M x^2 +Cms^2 d^2 \de^2 x^2}+e^{-x^2/2+\g x^3/6}].
}
Now let $M=C(ms^2 d^2 \de^2+1)$ for a sufficiently large $C$.
Note that from \eq{5} and \eq{47}, we have $1\leq 2m d^2\de^2$.
Recall $2d\de\leq 1$ from the assumption of the proposition.
The first term inside the brackets in \eq{51} is dominated by the second term, and
we have
\ben{\label{8}
|\text{LHS}|\leq \frac{C}{m} e^{x^2} (\varepsilon+d\de) e^{\varepsilon x} m s^2 d^2 \de^2 x \exp(-\frac{x^2}{2}+\frac{\gamma x^3}{6}).
}

\noindent\emph{Part II: Lower bound for RHS.}

Because $f$ is increasing and for $x-2d\de\leq w\leq x+\varepsilon+2d\de$,
\be{
f'(w)=xe^{xw}(w-x+2d\de)+e^{x w}\geq e^{x(x-2d\de)},
}
we have, from \eq{49} and \eq{60},
\bes{
\text{RHS}=& E(W-W')(f(W)-f(W'))\\
\geq & E (W-W')(f(W)-f(W')) I(x\leq W\leq x+\varepsilon) I(|W-W'|\leq 2d\de)\\
\geq & cE(W-W')^2 e^{x^2} I(x\leq W\leq x+\varepsilon).
}
Averaging over $I$, we have, recalling $V_2=\sum_{\alpha=1}^m(W-W^{\{\alpha\}})^2$,
\be{
\text{RHS}\geq \frac{c}{m}e^{x^2} E I(x\leq W\leq x+\varepsilon) V_2.
}
Recall from \eq{5} that $E V_2=C_2$.
We have
\bes{
\text{RHS}\geq& \frac{c C_2}{m}e^{x^2} E I(x\leq W\leq x+\varepsilon) I(V_2\geq C_2/2)\\
\geq &\frac{c C_2}{m}e^{x^2}  (P(x\leq W\leq x+\varepsilon)-P(V_2< C_2/2)).
}
We now find an upper bound for the second probability, which equals
\be{
P(EV_2-V_2>C_2/2).
}
Applying Lemma \ref{l7} and Lemma \ref{l10} to $EV_2-V_2$, we have
\be{
P(EV_2-V_2>C_2/2) \leq  e^{-\theta C_2/2} \exp(Cms^2d^4\de^4 \theta^2)= \exp(- \frac{1}{C ms^2d^4\de^4})
}
by choosing the optimal $\theta=C_2/4Cms^2 d^4 \de^4$ and using $C_2\geq 2$ from \eq{5}.
We have arrived at:
\ben{\label{9}
\text{RHS}\geq \frac{c}{m} e^{x^2} \big[ P(x\leq W\leq x+\varepsilon) -\exp(-\frac{ 1}{Cms^2d^4\de^4})    \big]
}
The proof is finished by combining \eq{8} and \eq{9}.

\end{proof}

\section{Proof of the main result}

In this section, we prove our main result, Theorem \ref{t1}.
The lemmas stated in the proof are proved below.
In this section, we use $C$ to denote positive constants and use $K$ to denote positive integers.
They can depend only on $C_0$ and
may differ in different expressions.
We use $O(a)$ to denote a quantity such that $|O(a)|\leq C a$.

\subsection{Proof of Theorem \ref{t1}}
First, we have the following absolute-error bound in normal approximation for $W$:
\begin{lemma}\label{l11}
\ben{\label{31}
\sup_{x\in \mathbb{R}} |P(W\leq x)-\Phi(x)|\leq C n s^2 d^2 \de^3
}
\end{lemma}

\medskip

From \eq{5} and \eq{47}, we have
\ben{\label{55}
|\g| x^2\leq Cns^2d^2 \de^3 x^2\leq C mn s^4 d^4 \de^5 x^2\leq C
}
for $x$ in \eq{25}.
If $x$ is bounded, from \eq{55}, we have
\ben{\label{56}
|(1-\Phi(x))-(1-\Phi(x))e^{\g x^3/6}|\leq C|\g|\leq C n s^2 d^2\de^3.
}
From \eq{31}, \eq{56} and \eq{55}, \eq{13} holds for bounded $x$.
Therefore, without loss of generality, we can assume in the following proof that $x$ is sufficiently large and $mns^4 d^4 \de^5$, and hence $|\g|$, is sufficiently small.
These conditions may be used implicitly below.

We only prove for the case $\g\ne 0$.
The case $\g=0$ follows from a similar and simpler proof by working directly with the standard normal distribution.
For $\g\ne 0$, $(1-\Phi(x))e^{\g x^3/6}$ is a no longer a distribution function.
We use a standardized Poisson distribution for an intermediate approximation.
Let $Z_\gamma=\gamma(Y_\gamma-\frac{1}{\gamma^2})$, where $Y_\gamma\sim Poi(\frac{1}{\gamma^2})$.
We have $E Z_\g=0, E Z_\g^2=1, E Z_\g^3=\gamma$.
From Cram\'er's expansion, see, for example, \cite[Chapter VIII, Theorem 2]{Pe75}, we have
\ben{\label{27}
\frac{P(Z_\g>x)}{(1-\Phi(x))e^{\g x^3/6}}=1+O(|\g|)(1+x)+O(\g^2) x^4\ \text{for}\ 0\leq x\leq C_0|\g|^{-1/2}
}
for $|\g|\leq 1$. Therefore, it suffices to prove
\besn{\label{57}
|P(W>x)-P(Z_\g>x)|\leq& C m n s^4 d^4 \de^5 (1+x^2)(1-\Phi(x)) e^{\g x^3/6}\\
\leq &C m n s^4 d^4 \de^5 x \exp(-\frac{x^2}{2}+\frac{\g x^3}{6}) .
}
Denote the support of $Z_\gamma$ by
\be{
\mathcal{S}=\{\gamma \mathbb{Z}^+-\frac{1}{\gamma}\}.
}
We denote 
\be{
\alpha=ns^2d^2\de^3
}
and use them interchangeably below.
Let
\be{
h_\alpha^+(w)=
\begin{cases}
1 & w< x,\\
1-2(\frac{w-x}{\alpha})^2 & x\leq w<x+\alpha/2,\\
2(1-\frac{w-x}{\alpha})^2 & x+\alpha/2\leq w<x+\alpha,\\
0 & w\geq x+\alpha.
\end{cases}
}
and
\be{
h_\alpha^-(w)=
\begin{cases}
1 & w< x-\alpha,\\
1-2(\frac{w-(x-\alpha)}{\alpha})^2 & x-\alpha\leq w<x-\alpha/2,\\
2(1-\frac{w-(x-\alpha)}{\alpha})^2 & x-\alpha/2\leq w<x,\\
0 & w\geq x.
\end{cases}
}
Let $h_\alpha=h_\alpha^+$ or $h_\alpha=h_\alpha^-$.
The following holds for either choice of $h_\alpha$.
It is straightforward to verify that $h_\alpha'$ exists and is continuous and
\be{
|h_\alpha'(w_1)|\leq \frac{2}{\alpha},\quad \Big|\frac{h_\alpha'(w_1)-h_\alpha'(w_2)}{w_1-w_2}\Big|\leq \frac{8}{\alpha^2},\ \forall\ w_1\ne w_2.
}
Note that
\besn{\label{38}
&E h_\alpha^-(W)-E h_\alpha^-(Z_\g)-P(x-\alpha<Z_\g\leq x)\\
\leq & P(W\leq x)-P(Z_\g\leq x)\\
\leq &E h_\alpha^+(W)-E h_\alpha^+(Z_\g)+P(x<Z_\g\leq x+\alpha).
}
For $w_0\in \mathcal{S}$, $|w_0|=O(|\g|^{-1/2})$ and sufficiently small $|\g|$, applying Stirling's approximation
and Taylor's expansion to the Poisson probability, we have
\ben{\label{28}
P(Z_\g=w_0)=P(Y_\g=\frac{w_0}{\g}+\frac{1}{\g^2})=\frac{|\g|}{\sqrt{2\pi}} \exp(-\frac{w_0^2}{2}+\frac{\g w_0^3}{6}+O(1)).
}
Therefore, the difference between $P(W\leq x)-P(Z_\g\leq x)$ and $E h_\alpha(W)-E h_\alpha(Z_\g)$ in \eq{38} is bounded by
\besn{\label{33}
P(x-\alpha<Z_\g\leq x+\alpha)=O(\alpha ) \exp(-\frac{x^2}{2}+\frac{\g x^3}{6}),
}
which is bounded by the right-hand side of \eq{57}.
To bound $Eh_\alpha (W)-E h_\alpha(Z_\g)$,
consider the Stein equation for $Z_\gamma$:
\ben{\label{19}
\frac{1}{\gamma} (f(w+\gamma)-f(w))-wf(w)=h_\alpha(w)-E h_\alpha(Z_\g).
}
It has the following solution $f:=f_{h_\alpha}$ on $\mathcal{S}$:
$f(-1/\g)=0$ and
for $w_0\in\mathcal{S}\backslash \{-\frac{1}{\gamma}\}$,
\besn{\label{24}
f(w_0)=&\frac{1}{\frac{1}{\gamma}P(Y_\gamma= \frac{1}{\g^2}+\frac{w_0}{\g}-1)} E[h_\alpha(Z_\g)-E h_\alpha(Z_\g)]I(Y_\gamma\leq \frac{1}{\g^2}+\frac{w_0}{\g}-1)\\
=& - \frac{1}{\frac{1}{\gamma}P(Y_\gamma= \frac{1}{\g^2}+\frac{w_0}{\g})} E[h_\alpha(Z_\g)-E h_\alpha(Z_\g)]I(Y_\gamma\geq \frac{1}{\g^2}+\frac{w_0}{\g}),
}
where we recall that $Y_\g\sim \text{Poi}(1/\g^2)$.
From the expression of $f$ in \eq{24} and $|h_\alpha(Z_\g)-E h_\alpha(Z_\g)|\leq 1$, we have
\be{
|f(w_0)|\leq \frac{1}{|\g|}\min\{\frac{\g^2 P(Y_\gamma\leq \frac{1}{\g^2}+\frac{w_0}{\g}-1)}{P(Y_\gamma= \frac{1}{\g^2}+\frac{w_0}{\g}-1)}, \frac{\g^2 P(Y_\gamma\geq \frac{1}{\g^2}+\frac{w_0}{\g})}{P(Y_\gamma= \frac{1}{\g^2}+\frac{w_0}{\g})} \}.
}
From the proof of Lemma 1.1.1 of \cite{BaHoJa92} (cf. (1.20) and (1.21) therein),
if $w_0\leq \g$, then the first term inside the minimum is bounded by $2(1\wedge |\g|)$,
and if $w_0> \g$, then the second term inside the minimum is bounded by $2(1\wedge |\g|)$.
Therefore,
\ben{\label{32}
|f(w_0)|\leq 2.
}
Because in general our $W$ has different support from $\mathcal{S}$,
we extend $f$ to $f: \mathbb{R}\to \mathbb{R}$ as follows.
Let $f(w_0)=0$ for $w_0\in \{\gamma \mathbb{Z}^{-}-\frac{1}{\gamma}\}$.
For $w$ between $w_0$ and $w_0+\g$ such that $w_0\in \{\gamma \mathbb{Z}-\frac{1}{\gamma}\}$,
we define $f(w)$ to be a fifth-order polynomial function such that it matches the discrete derivatives at $w_0$ and $w_0+\g$ up to the second order.
In more detail, let
\be{
f_0:=f(w_0),\ f_1:=f(w_0+\gamma),\ f_2:=f(w_0+2\gamma),\ f_{-1}:=f(w_0-\gamma),
}
\ben{\label{39}
f_0':=\frac{f_1-f_{-1}}{2\gamma},\ f_1':=\frac{f_2-f_0}{2\gamma},
}
\ben{\label{40}
f_0'':=\frac{f_1-2f_0+f_{-1}}{\gamma^2},\ f_1'':=\frac{f_2-2f_1+f_0}{\gamma^2},
}
and let
\ben{\label{41}
f(w)=\sum_{i=1}^6 b_i (w-w_0)^{6-i},
}
where
\bes{
&b_1=-\frac{1}{\gamma^2}\cdot \frac{f_1''-f_0''}{\gamma},\\
&b_2=\frac{5}{2\gamma}\cdot \frac{f_1''-f_0''}{\gamma},\\
&b_3=-\frac{3}{2}\cdot \frac{f_1''-f_0''}{\gamma},\\
&b_4=\frac{f_0''}{2},\  b_5=f_0',\  b_6=f_0.
}
In the following, for any $w\in \mathbb{R}$, let $w_0$ be such that $w_0\in \{\gamma \mathbb{Z}-\frac{1}{\gamma}\}$ and
$w_0+\g < w\leq w_0$ if $\g<0$ and
$w_0\leq w< w_0+\gamma$ if $\g>0$.
For a random variable $W$, $W_0$ is defined in the same way as for $w_0$.

It follows from the construction of $f$ above that $f''(w)$ exists and is continuous and $f^{(3)}(w)$ exists for $w\notin \mathcal{S}$.
For $w\in \mathcal{S}$, we define $f^{(3)}(w)=0$ as they will not enter into consideration when we do Taylor's expansion below (cf. \eq{44}).
Note that
\ben{\label{42}
f(w)=O(1)(f(w_0-\g)+f(w_0)+f(w_0+\g)+f(w_0+2\g)).
}
Therefore, from \eq{32}, $f$ is bounded.
Note that after such extension, $f$ \emph{no longer} satisfies \eq{19} exactly, except on $\mathcal{S}$.
However, we can quantify the error as in the following lemma.
\begin{lemma}\label{l6}
For the above defined $f$, we have,
\besn{\label{20}
		&\frac{1}{\gamma}(f(w+\gamma)-f(w))-w f(w)\\
		=&h_\alpha(w)-E h_\alpha(Z_\g)+O(1)I(|w-x|\leq C\alpha)\\
		&+O(|\g|)\sum_{i=-K}^{K}|f(w_0+i\cdot\gamma)| +O(1) I(w<-1/\gamma+\gamma)I(\g>0)\\
		&+O(1) I(w>-1/\gamma+\gamma)I(\g<0).
}
\end{lemma}

By replacing $w$ by $W$ and $w_0$ by $W_0$ in \eq{20} and taking expectations on both sides, we have
\bes{
		&E h_\alpha(W)-E h_\alpha(Z_\g)\\
		=&E[\frac{1}{\gamma}(f(W+\gamma)-f(W))-W f(W)]+O(1)P(|W-x|\leq C\alpha)\\
		& \quad+O(|\gamma|)\sum_{i=-K}^{K}E |f(W_0+i\cdot\gamma)| +O(1) P(W<-1/\gamma+\gamma)I(\g>0)\\
		&\quad +O(1) P(W>-1/\gamma+\gamma)I(\g<0)\\
		=:&R_1+R_2+R_3+R_4+R_5.
}
We bound these remainders in the reverse order.
If $\g>0$, we have, by applying Proposition \ref{p1} to $-W$,
\besn{\label{43}
&P(W<-1/\gamma+\gamma)\leq  C E e^{-xW}/e^{x/\g}\\
\leq& C \exp(\frac{x^2}{2}-\frac{\g x^3}{6}-\frac{x}{|\g|})\\
=& C |\g| \exp(-\frac{x^2}{2}+\frac{\g x^3}{6}) \frac{1}{|\g|}\exp(x^2-\frac{\g x^3}{3}-\frac{x}{|\g|})\\
\leq & C|\g| \exp(-\frac{x^2}{2}+\frac{\g x^3}{6}),
}
where we use $1\leq x=O(1)|\g|^{-1/2}$ and $|\g|$ is sufficiently small (cf. the arguments below \eq{56}).
Together with the same bound for $R_5$, we have
\be{
|R_4|+|R_5|\leq C |\g| \exp(-\frac{x^2}{2}+\frac{\g x^3}{6}).
}
To bound $R_3$, we use the following lemma.
\begin{lemma}\label{l3}
We have,
\besn{\label{34}
\sum_{i=-K}^K E|f(W_0+i\g)|\leq &C P(W\geq x-C\alpha)+C P(Z_\g>x-\alpha)\\
&+C E I(0\leq W\leq x)e^{\frac{W^2}{2}-\frac{\g W^3}{6}}P(Z_\g>x-\alpha).
}
\end{lemma}

For the first term on the right-hand side of \eq{34}, we have, by \eq{4},
\bes{
&P(W\geq x-C\alpha)\leq C e^{-x^2} \exp(\frac{x^2}{2}+\frac{\gamma x^3}{6})\\
=& C\exp(-\frac{x^2}{2}+\frac{\gamma x^3}{6}).
}
For the second term on the right-hand side of \eq{34}, we have, by \eq{27},
\be{
P(Z_\g\geq x-\alpha)\leq \frac{C}{x}\exp(-\frac{x^2}{2}+\frac{\gamma x^3}{6}).
}
For the third term on the right-hand side of \eq{34}, we have
\besn{\label{35}
&E I(0\leq W\leq x)e^{\frac{W^2}{2}-\frac{\g W^3}{6}}\\
\leq & 1+C\int_0^x (y+1) e^{\frac{y^2}{2}-\frac{\g y^3}{6}} P(W>y) dy\\
=&O(x),
}
where we use the following lemma in the last step.

\begin{lemma}\label{l9}
For integer $k\geq 1$, we have
\be{
\int_0^x y^k e^{\frac{y^2}{2}-\frac{\g y^3}{6}} P(W>y) dy=O(1) x^k.
}
\end{lemma}

Combining these bounds, we have
\bes{
|R_3|\leq C|\g| \exp(-\frac{x^2}{2}+\frac{\g x^3}{6}).
}
Next, we use Proposition \ref{p2} to bound $R_2$ as follows.
Recall we assumed without loss of generality that $x$ is sufficiently large, $mns^4d^4\de^5$, and hence $\alpha$, is sufficiently small (cf. \eq{47}). We have, from Proposition \ref{p2} and $d\de\leq \alpha$ (cf. \eq{47}),
\besn{\label{58}
|R_2|\leq &C P(|W-x|\leq C\alpha)\\
\leq & C ms^2d^2\de^2 \alpha  x \exp(-\frac{x^2}{2}+\frac{\g x^3}{6}) +\exp(-\frac{1}{C ms^2d^4\de^4})\\
\leq & C ms^2d^2\de^2 \alpha x\exp(-\frac{x^2}{2}+\frac{\g x^3}{6})\\
&+ C\exp(-\frac{1}{C ms^2d^4\de^4}+C x^2)\exp(-\frac{x^2}{2}+\frac{\g x^3}{6}).
}
Note that for $x=O(mns^4d^4\de^5)^{-1/2}$, we have (cf. \eq{47})
\be{
ms^2d^4\de^4 x^2=O(\frac{ms^2d^4\de^4}{mns^4d^4\de^5})=O(\frac{ns^2d^2\de^3}{n^2s^4d^2\de^4})
=O(ns^2d^2\de^3).
}
Therefore, the second term on the right-hand side of \eq{58} is dominated by the first term and
\ben{\label{59}
|R_2|\leq C m n s^4 d^4 \de^5 x \exp(-\frac{x^2}{2}+\frac{\g x^3}{6}).
}
We are now left to bound $R_1$.
By Taylor's expansion and exploiting the local dependence structure (LD1)--(LD3) in Section 3.1, we have the following lemma.
Note that this is where we use the crucial choice of $Z_\g$ so that it matches the moments of $W$ up to the third order.

\begin{lemma}\label{l2}
We have
\bes{
&E[\frac{1}{\gamma}(f(W+\gamma)-f(W))-W f(W)]\\
=&O(\alpha^2) E |f^{(3)} (W+O(\alpha))|.
}
\end{lemma}

To bound $f^{(3)}$, we use the following lemma.

\begin{lemma}\label{l4}
We have
\besn{\label{36}
&E|f^{(3)}(W+O(\alpha))|\\
\leq &\frac{C}{\alpha^2} P(|W-x|\leq C\alpha)\\
&+C E (1+|W|^3) |f(W+O(\alpha))|\\
&+C E (1+W^2) I(W\geq x-C\alpha)\\
&+C P(Z_\g>x-C\alpha)\\
&+ \frac{C}{\g^2} P(W\leq -1/\g+O(\alpha))I(\g>0)\\
&+ \frac{C}{\g^2} P(W\geq -1/\g-O(\alpha))I(\g<0).
}
\end{lemma}

The first term on the right-hand side of \eq{36} is bounded as in \eq{58} and \eq{59}.
For the second term on the right-hand side of \eq{36},
from the proof of Lemma \ref{l3}, we have
\besn{\label{37}
&E(1+|W|^3)|f(W+O(\alpha))|\\
\leq &C E (1+W^3) I(W\geq x-C\alpha)\\
&+ C P(Z_\g>x-\alpha) (1+E|W^3|)\\
&+ C E(1+W^3)I(0\leq W\leq x) e^{\frac{W^2}{2}-\frac{\g W^3}{6}} P(Z_\g>x-\alpha).
}
Similar to \eq{35}, using Lemma \ref{l9}, we have
\bes{
&E (1+W^3)I(0\leq W\leq x)e^{\frac{W^2}{2}-\frac{\g W^3}{6}}\\
\leq & 1+ \int_0^x [3y^2+(1+y^3)(y-\frac{\g y^2}{2})] e^{\frac{y^2}{2}-\frac{\g y^3}{6}} P(W>y) dy\\
=& O(1)\int_0^x (1+y^4) e^{\frac{y^2}{2}-\frac{\g y^3}{6}} P(W>y) dy\\
=& O(1) x^4.
}
For the third term on the right-hand side of \eq{36} and the first term on the right-hand side of \eq{37}, we have
\begin{lemma}\label{l8}
\bes{
E(1+W^3)I(W\geq x-C\alpha)=O(1) x^3 \exp(-\frac{x^2}{2}+\frac{\g x^3}{6}).
}
\end{lemma}

Note that
\be{
E|W|^3\leq \sqrt{E W^4}\leq C\sqrt{1+ns^3 d^3\de^4}\leq C(1+n s^2 d^2 \de^3)\leq C.
}
The fourth term on the right-hand side of \eq{36} and the second term on the right-hand side of \eq{37} are bounded from \eq{27} by
\be{
CP(Z_\g>x-C\alpha)\leq \frac{C}{x}\exp(-\frac{x^2}{2}+\frac{\g x^3}{6}).
}
The fifth and sixth terms on the right-hand side of \eq{36} are bounded in a manner similar as for $R_4$ (cf. \eq{43}) by
\bes{
&\frac{C}{\g^2} P(W\leq -1/\g+O(\alpha))I(\g>0)\\
&+\frac{C}{\g^2} P(W\geq -1/\g-O(\alpha))I(\g<0)\\
 \leq& C|\g|\exp(-\frac{x^2}{2}+\frac{\g x^3}{6}).
}
In summary, we have
\bes{
\frac{|R_1|}{\exp(-x^2/2+\g x^3/6)}\leq & C \alpha^2 \big( \frac{1}{\alpha^2} m n s^4 d^4 \de^5 x +x^3  \big)\\
\leq & C m n s^4 d^4 \de^5 x,
}
where we use $n^2 s^4 d^4 \de^6 x^2\leq C n s^2 d^2 \de^3\leq C m n s^4 d^4 \de^5$ (cf. \eq{47}).
The bound \eq{57}, hence the theorem, is proved by combining \eq{33} and the bounds on $|R_1|$--$|R_5|$.

\subsection{Proofs of lemmas}

In the following, we prove the lemmas stated in the proof above.
\begin{proof}[Proof of Lemma \ref{l11}]
Denote the Kolmogorov distance between two probability distributions by
\be{
d_K(\mathcal{L}(X), \mathcal{L}(Y)):=\sup_{x\in \mathbb{R}}|P(X\leq x)-P(Y\leq x)|.
}
For $\beta>0$ to be chosen, let
\be{
g_\beta(w)=
\begin{cases}
1 & w\leq x,\\
1+(x-w)/\beta & x<w\leq x+\beta,\\
0 & w>x+\beta.
\end{cases}
}
Let $F:=F_{g_\beta}$ be the bounded solution to
\ben{\label{16}
F'(w)-w F(w)=g_\beta(w)-E g_\beta (Z),
}
where $Z\sim N(0,1)$.
From Lemma 2.5 of \cite{ChGoSh11}, we have
\ben{\label{15}
|F'(w+v)-F'(w)|\leq |v|\left(1+|w|+\frac{1}{\beta} \int_0^1 I_{[x, x+\beta]}(w+rv) dr \right) .
}
Replacing $w$ by $W$ and taking expectations on both sides of the equation \eq{16}, we have
\besn{\label{17}
&P(W\leq x)-\Phi(x)\leq E g_\beta(W)-E g_\beta(Z)+Eg_\beta(Z)-\Phi(x)\\
\leq& EF'(W)-EWF(W)+P(x\leq Z\leq x+\beta)\\
\leq& EF'(W)-EWF(W)+C\beta.
}
Let $U\sim \text{Unif}[0,1]$ be independent of all else.
By (LD1), (LD2), $E\xi_i=0$, and Taylor's expansion, we have
\bes{
E W F(W)=& \sum_{i=1}^n E \xi_i F(W) =\sum_{i=1}^n E \xi_i [F(W)-F(W-\xi_{A_i})]\\
=& \sum_{i=1}^n E \xi_i \xi_{A_i} F'(W-U \xi_{A_i})\\
=&\sum_{i=1}^n \sum_{j\in A_i} E \xi_i \xi_j E F'(W-\xi_{A_{ij}})\\
&+\sum_{i=1}^n \sum_{j\in A_i} E \xi_i \xi_j [F'(W-U\xi_{A_i})- F'(W-\xi_{A_{ij}})].
}
From $E W^2=\sum_{i=1}^n \sum_{j\in A_i} E \xi_i \xi_j=1$, we have
\bes{
&EF'(W)-E W F(W)\\
=& \sum_{i=1}^n \sum_{j\in A_i} E \xi_i \xi_j E[F'(W)- F'(W-\xi_{A_{ij}})]\\
&-\sum_{i=1}^n \sum_{j\in A_i} E \xi_i \xi_j [F'(W-U\xi_{A_i})- F'(W-\xi_{A_{ij}})].
}
From \eq{15} and the boundedness conditions in \eq{11} and \eq{48}, we have
\bes{
&|EF'(W)-E W F(W)|\\
\leq & Cn s^2 d^2 \de^3 (1+\frac{1}{\beta} P(W\in [x-Csd \de, x+\beta+Csd\de])).
}
Using
\be{
P(W\in [x-Csd \de, x+\beta+Csd\de]))\leq 2d_K(\mathcal{L}(W), N(0,1))+C(sd\de+\beta),
}
we have
\besn{\label{18}
&|EF'(W)-E W F(W)|\\
\leq & Cn s^2d^2 \de^3 +\frac{C n s^2d^2 \de^3 sd \de}{\beta}+Cn s^2 d^2 \de^3 \frac{d_K(\mathcal{L}(W), N(0,1))}{\beta}.
}
From \eq{17} and \eq{18}, we have
\bes{
&P(W\leq x)-\Phi(x)\\
\leq & C\beta+Cn s^2 d^2 \de^3 +\frac{C n s^2 d^2 \de^3 sd \de}{\beta}+Cn s^2 d^2 \de^3 \frac{d_K(\mathcal{L}(W), N(0,1))}{\beta}.
}
From a similar argument for the lower bound, we have
\bes{
&|P(W\leq x)-\Phi(x)|\\
\leq & C\beta+Cn s^2 d^2 \de^3 +\frac{C n s^2 d^2 \de^3 sd \de}{\beta}+Cn s^2 d^2 \de^3 \frac{d_K(\mathcal{L}(W), N(0,1))}{\beta}.
}
Taking supremum over $x$, choosing $\beta=2C n s^2 d^2 \de^3$, solving the resulting recursive inequality for $d_K(\mathcal{L}(W), N(0,1))$, and noting that $sd\de\leq ns^2d^2\de^3$ from \eq{47}, we arrive at
\be{
d_K(\mathcal{L}(W), N(0,1))\leq C n s^2 d^2 \de^3.
}
\end{proof}

\begin{proof}[Proof of Lemma \ref{l6}]
We only prove for the case $\g>0$. The case $\g<0$ can be proved similarly.
For $w<-\frac{1}{\g}+\g$, because
\be{
f(-1/\g)=0,
}
\be{
f(-1/\g+\g)=\g (1-E h_\alpha(Z_\g)),
}
\be{
f(-1/\g+2\g)=\g^2 f(-1/\g+\g) + \g (1-E h_\alpha(Z_\g))
}
\be{
f(-1/\g+3\g)=2\g^2 f(-1/\g+2\g) + \g (1-E h_\alpha(Z_\g)),
}
we have (cf. \eq{42})
\be{
\frac{1}{\gamma}(f(w+\gamma)-f(w))-w f(w)=O(1).
}

For $-\frac{1}{\gamma}+\gamma\leq w$ and $w_0\leq w<w_0+\g$ such that $w_0\in \{\gamma \mathbb{Z}-\frac{1}{\gamma}\}$, we have, from the construction of $f$ (cf. \eq{41}),
\bes{
& \frac{1}{\gamma}(f(w+\gamma)-f(w))-w f(w)\\
=&\frac{1}{\g} [f(w_0+\g)-f(w_0)]-wf(w_0)\\
&+(w-w_0) \Big\{ \frac{1}{\g} [f'(w_0+\g)-f'(w_0)]-wf'(w_0)   \Big\} \\
&+\frac{(w-w_0)^2}{2}  \Big\{ \frac{1}{\g} [f''(w_0+\g)-f''(w_0)]-wf''(w_0)   \Big\}\\
&+[-\frac{3(w-w_0)^3}{2\g}+\frac{5(w-w_0)^4}{2\g^2}-\frac{(w-w_0)^5}{\g^3} ]\\
&\quad \times\Big\{ \frac{1}{\g} [(f''(w_0+2\g)-f''(w_0+\g))-(f''(w_0+\g)-f''(w_0))]-w(f''(w_0+\g)-f''(w_0))   \Big\}\\
=:&H_1+H_2+H_3+H_4.
}
Note that $f$ satisfies \eq{19} on $\mathcal{S}$. We have
\bes{
H_1=&\frac{1}{\g} [f(w_0+\g)-f(w_0)]-w_0f(w_0)-(w-w_0)f(w_0)\\
=&h_\alpha(w_0)-E h_\alpha(Z_\g)+O(|\g|)\sum_{i=-K}^{K}|f(w_0+i\cdot\gamma)|\\
=&h_\alpha(w)-E h_\alpha(Z_\g)+O(1)I(|w-x|\leq C\alpha)+O(|\gamma|)\sum_{i=-K}^{K}|f(w_0+i\cdot\gamma)| .
}
For $H_2$, from the expression of $f'$ on $\mathcal{S}$ (cf. \eq{39}) and using again the fact that $f$ satisfies \eq{19} on $\mathcal{S}$, we have
\bes{
H_2=& \frac{w-w_0}{2\g} \Big\{  \frac{1}{\g} \big[ (f(w_0+2\g)-f(w_0))-(f(w_0+\g)-f(w_0-\g))  \big]\\
&\qquad \qquad  -w(f(w_0+\g)-f(w_0-\g))       \Big\}\\
=&O(1)[h_\alpha(w_0+\g)-h_\alpha(w_0-\g)]+O(|\gamma|)\sum_{i=-K}^{K}|f(w_0+i\cdot\gamma)|\\
=&O(1)I(|w-x|\leq C\alpha)+O(|\gamma|)\sum_{i=-K}^{K}|f(w_0+i\cdot\gamma)|.
}
Similarly, from \eq{40},
\bes{
H_3=& \frac{(w-w_0)^2}{2\g^2} \Big\{  \frac{1}{\g} \big[ (f(w_0+2\g)-2f(w_0+\g)+f(w_0))-(f(w_0+\g)-2f(w_0)+f(w_0-\g))  \big]\\
&\qquad \qquad  -w(f(w_0+\g)-2f(w_0)+f(w_0-\g))       \Big\}\\
=&O(1)I(|w-x|\leq C\alpha)+O(|\gamma|)\sum_{i=-K}^{K}|f(w_0+i\cdot\gamma)|,
}
and
\be{
H_4=O(1)I(|w-x|\leq C\alpha)+O(|\gamma|)\sum_{i=-K}^{K}|f(w_0+i\cdot\gamma)|.
}
Equation \eq{20} is proved by combining the above estimates and observing that the right-hand side is bounded.
\end{proof}

\begin{proof}[Proof of Lemma \ref{l3}]
We only prove for the case $\g>0$. The case $\g<0$ can be proved similarly.
Recall the definition of $f$.
If $w_0-\g>x-\alpha$, then we use $|f(w_0)|\leq C$.
If $w_0\leq -1/\g$, then $f(w_0)=0$.
If $-1/\g\leq  w_0-\g\leq x-\alpha$, then
\be{
f(w_0)=\frac{\g P(Z_\g\leq w_0-\g)}{P(Z_\g=w_0-\g)}[1-E h_\alpha(Z_\g)].
}
Recall the proof of \eq{32}, if $-1/\g\leq w_0-\g\leq 0$, then
\be{
0\leq f(w_0)\leq 2 P(Z_\g>x-\alpha).
}
If $0<w_0-\g\leq x-\alpha$, then by \eq{28},
\be{
|f(w_0)|\leq C e^{\frac{w_0^2}{2}-\frac{\g w_0^3}{6}} P(Z_\g>x-\alpha).
}
The lemma is proved by combining the above bounds and noting that $|W-W_0|\leq \g$.
\end{proof}

\begin{proof}[Proof of Lemma \ref{l9}]
Similar to the proof of Lemma 5.2 of \cite{ChFaSh13} and use \eq{4}, we have, for some $\epsilon\in [0,1]$,
\bes{
&\int_0^{[x]} y^k e^{\frac{y^2}{2}-\frac{\g y^3}{6}} P(W>y) dy\\
\leq & \sum_{j=1}^{[x]} j^k \int_{j-1}^j e^{\frac{y^2}{2}-\frac{\g y^3}{6}-jy} e^{jy} P(W>y ) dy\\
\leq & \sum_{j=1}^{[x]} j^k e^{\frac{(j-1)^2}{2}-\frac{\g (j-\epsilon)^3}{6}-j(j-1)} \int_{j-1}^j e^{jy} P(W>y) dy\\
\leq & 2 \sum_{j=1}^{[x]} j^k e^{-\frac{j^2}{2}-\frac{\g (j-\epsilon)^3}{6}} \int_{-\infty}^{\infty} e^{jy} P(W>y) dy\\
=& 2 \sum_{j=1}^{[x]} j^k e^{-\frac{j^2}{2}-\frac{\g (j-\epsilon)^3}{6}} \frac{1}{j} E e^{jW}\\
=& O(1) \sum_{j=1}^{[x]} j^{k-1} e^{\frac{\g j^3}{6}-\frac{\g (j-\epsilon)^3}{6}}=O(1) x^k.
}
Similarly, we have
\bes{
& \int_{[x]}^x y^k e^{\frac{y^2}{2}-\frac{\g y^3}{6}} P(W>y) dy\\
\leq & x^k \int_{[x]}^x e^{\frac{y^2}{2}-\frac{\g y^3}{6}-xy} e^{xy} P(W>y) dy\\
\leq & x^k e^{\frac{[x]^2}{2}-\frac{\g ([x]+\epsilon)^3}{6}-x[x]} \int_{[x]}^x e^{xy} P(W>y) dy\\
\leq & 2x^k e^{-\frac{x^2}{2}-\frac{\g ([x]+\epsilon)^3}{6}} \int_{-\infty}^\infty e^{xy} P(W>y) dy\\
=& O(1) x^k.
}
This finishes the proof.
\end{proof}

\begin{proof}[Proof of Lemma \ref{l2}]
Let $U_1, U_2$ be independent $\sim$Unif$[0,1]$ and independent of all else.
By Taylor's expansion,
\besn{\label{44}
&E \frac{1}{\gamma} (f(W+\gamma)-f(W))\\
=& E \frac{1}{\gamma} [\gamma f'(W)+\frac{\gamma^2}{2} f''(W) +\gamma^3 (1-U_2)U_2 f^{(3)} (W+\g U_1 U_2)]\\
=& E f'(W) +\frac{\gamma}{2} E f''(W) + O(\gamma^2) E |f^{(3)} (W+O(|\gamma|))|.
}
By the local dependence structure (LD1)--(LD3) in Section 3.1, $E\xi_i=0$, Taylor's expansion and the boundedness conditions in \eq{11} and \eq{48}, we have
\bes{
&E Wf(W)\\
=& \sum_{i=1}^n E \xi_i [f(W)-f(W-\xi_{A_i})]\\
=& \sum_{i=1}^n E \xi_i [\xi_{A_i} f'(W-\xi_{A_i})+\frac{\xi_{A_i}^2}{2} f''(W-\xi_{A_i}) +O(s^3 d^3 \de^3) |f^{(3)}(W+O(sd\de))|]\\
=& \sum_{i=1}^n \sum_{j\in A_i} E \xi_i \xi_j f'(W-\xi_{A_i}) + \frac{1}{2} \sum_{i=1}^n \sum_{j\in A_i} \sum_{k\in A_i}
E \xi_i \xi_j \xi_k f''(W-\xi_{A_i})\\
&+O(ns^3 d^3 \de^4) E |f^{(3)} (W+O(sd \de))|\\
=:& B_1+B_2+O(ns^3d^3 \de^4) E |f^{(3)} (W+O(sd \de))|.
}
For $B_1$, by a similar expansion as above, we have
\bes{
& \sum_{i=1}^n \sum_{j\in A_i} E \xi_i \xi_j f'(W-\xi_{A_i})\\
=& \sum_{i=1}^n \sum_{j\in A_i} E \xi_i \xi_j E f'(W-\xi_{A_{ij}})+ \sum_{i=1}^n \sum_{j\in A_i} E \xi_i \xi_j [f'(W-\xi_{A_i})-f'(W-\xi_{A_{ij}})]\\
=& E f'(W) +\sum_{i=1}^n \sum_{j\in A_i} E \xi_i \xi_j E[f'(W-\xi_{A_{ij}})-f'(W)]\\
&+ \sum_{i=1}^n \sum_{j\in A_i} E \xi_i \xi_j [f'(W-\xi_{A_i})-f'(W-\xi_{A_{ij}})]\\
=&E f'(W)- \sum_{i=1}^n \sum_{j\in A_i} E \xi_i \xi_j E\xi_{A_{ij}} f''(W-\xi_{A_{ij}}) +O(ns^3 d^3 \de^4) E|f^{(3)} (W+O(sd \de))|\\
&+\sum_{i=1}^n \sum_{j\in A_i} E \xi_i \xi_j (\xi_{A_{ij}}-\xi_{A_i})f''(W-\xi_{A_{ij}}).
}
For $B_2$, we have
\bes{
&\frac{1}{2} \sum_{i=1}^n \sum_{j\in A_i} \sum_{k\in A_i}
E \xi_i \xi_j \xi_k f''(W-\xi_{A_i})\\
=& \frac{1}{2} \sum_{i=1}^n \sum_{j\in A_i} \sum_{k\in A_i}
E \xi_i \xi_j \xi_k E f''(W-\xi_{A_{ijk}})\\
&+\frac{1}{2} \sum_{i=1}^n \sum_{j\in A_i} \sum_{k\in A_i}
E \xi_i \xi_j \xi_k [f''(W-\xi_{A_i})-f''(W-\xi_{A_{ijk}})]\\
=& \frac{1}{2} \sum_{i=1}^n \sum_{j\in A_i} \sum_{k\in A_i}
E \xi_i \xi_j \xi_k E f''(W)\\
&+ \frac{1}{2} \sum_{i=1}^n \sum_{j\in A_i} \sum_{k\in A_i}
E \xi_i \xi_j \xi_k E[f''(W-\xi_{A_{ijk}})-f''(W)]\\
&+\frac{1}{2} \sum_{i=1}^n \sum_{j\in A_i} \sum_{k\in A_i}
E \xi_i \xi_j \xi_k [f''(W-\xi_{A_i})-f''(W-\xi_{A_{ijk}})]\\
=& \frac{1}{2} \sum_{i=1}^n \sum_{j\in A_i} \sum_{k\in A_i}
E \xi_i \xi_j \xi_k E f''(W)+O(ns^3 d^3 \de^4) E |f^{(3)} (W+O(sd \de))|.
}
Similarly,
\bes{
&- \sum_{i=1}^n \sum_{j\in A_i} E \xi_i \xi_j E\xi_{A_{ij}} f''(W-\xi_{A_{ij}}) \\
=& - \sum_{i=1}^n \sum_{j\in A_i} \sum_{k\in A_{ij}} E \xi_i \xi_j E\xi_{k} [f''(W-\xi_{A_{ij}})- f''(W-\xi_{A_{ijk}}) ]\\
=&O(ns^3 d^3 \de^4) E |f^{(3)} (W+O(sd \de))|,
}
and
\bes{
& \sum_{i=1}^n \sum_{j\in A_i} E \xi_i \xi_j (\xi_{A_{ij}}-\xi_{A_i})f''(W-\xi_{A_{ij}})    \\
=& \sum_{i=1}^n \sum_{j\in A_i} E \xi_i \xi_j (\xi_{A_{ij}}-\xi_{A_i})f''(W) + O(ns^3 d^3 \de^4) E |f^{(3)} (W+O(sd \de))|.
}
Recall $\g$ from \eq{50}.
Combining the above estimates, we have
\ben{\label{45}
E Wf(W)=E f'(W)+\frac{\gamma}{2} E f''(W)+O(ns^3 d^3 \de^4) E |f^{(3)} (W+O(sd \de))|.
}
By \eq{44} and \eq{45}, we conclude that
\bes{
&E[\frac{1}{\gamma}(f(W+\gamma)-f(W))-W f(W)]\\
=&O(\gamma^2) E f^{(3)} (W+O(|\g|))+ O(ns^3 d^3 \de^4) E |f^{(3)} (W+O(sd \de))|\\
=& O(n^2 s^4 d^4 \de^6) E|f^{(3)}(W+O(ns^2 d^2 \de^3))|,
}
where we use $|\g|\leq C n s^2 d^2 \de^3$, $sd\de\leq C ns^2 d^2 \de^3$ and $ns^3 d^3 \de^4\leq C(n s^2 d^2 \de^3)^2$ from \eq{47}.
\end{proof}

\begin{proof}[Proof of Lemma \ref{l4}]
We only prove for the case $\g>0$. The case $\g<0$ can be proved similarly.
Note that from the construction of $f$ (cf. \eq{41}),
\be{
f^{(3)}(w)=O(1)\frac{f''(w_0+\g)-f''(w_0)}{\g}.
}
For $w_0\leq -1/\g$,
from the arguments at the beginning of the proof of Lemma \ref{l6},
$\frac{f''(w_0+\g)-f''(w_0)}{\g}=O(1/\g^2)$.
For $w_0\geq -1/\g+\g$, from the construction of $f$ (cf. \eq{40}) and the equation \eq{19} for $w\in \mathcal{S}$, we have
\besn{\label{21}
& \frac{f''(w_0+\g)-f''(w_0)}{\g}\\
=& \frac{[f(w_0+2\g)-2f(w_0+\g)+f(w_0)]-[f(w_0+\g)-2f(w_0)+f(w_0-\g)]}{\g^3}\\
=& \frac{[(w_0+\g)f(w_0+\g)+h_\alpha(w_0+\g)]-[w_0 f(w_0)+h_\alpha(w_0)]}{\g^2}\\
&-\frac{[w_0f(w_0)+h_\alpha(w_0)]-[(w_0-\g) f(w_0-\g)+h_\alpha(w_0-\g)]}{\g^2}.
}
Rearranging terms, using $|h_\alpha''(w)|\leq \frac{8}{\alpha^2}I(x-\alpha\leq w\leq x+\alpha)$ and that $f$ solves \eq{19} on $\mathcal{S}$, we have
\bes{
\eq{21}=& O(\frac{1}{\alpha^2})I(|w_0-x|\leq C\alpha)\\
&+\frac{1}{\g}[f(w_0+\g)-f(w_0)]+\frac{w_0}{\g^2} [f(w_0+\g)-f(w_0)]\\
&-\frac{w_0-\g}{\g^2}[f(w_0)-f(w_0-\g)]\\
=&O(\frac{1}{\alpha^2})I(|w_0-x|\leq C\alpha)\\
&+[w_0 f(w_0)+h_\alpha(w_0)-E h_\alpha(Z_\g)](1+\frac{w_0}{\g}) \\
&-\frac{w_0-\g}{\g}[(w_0-\g) f(w_0-\g)+h_\alpha(w_0-\g)-E h_\alpha(Z_\g)].
}
Rearranging terms, using $|h_\alpha'(w)|\leq \frac{2}{\alpha}I(x-\alpha \leq w\leq x+\alpha)$ and that $f$ solves \eq{19} on $\mathcal{S}$, we have
\bes{
\eq{21}=&O(\frac{1}{\alpha^2})I(|w_0-x|\leq C\alpha)\\
&+w_0 f(w_0)+(w_0-\g) f(w_0-\g)+w_0 f(w_0-\g)\\
&+ h_\alpha(w_0)-E h_\alpha(Z_\g) +h_\alpha(w_0-\g) -E h_\alpha(Z_\g)\\
&+\frac{w_0}{\g}(f(w_0)-f(w_0-\g))+\frac{w_0}{\g} (h_\alpha(w_0)-h_\alpha(w_0-\g))\\
=& O(\frac{1}{\alpha^2})I(|w_0-x|\leq C\alpha)\\
&+w_0 f(w_0)+(w_0-\g) f(w_0-\g)+w_0 f(w_0-\g)\\
&+ h_\alpha(w_0)-E h_\alpha(Z_\g) +h_\alpha(w_0-\g) -E h_\alpha(Z_\g)\\
&+ O(\frac{1}{\alpha}) x I(|w_0-x|\leq C\alpha)\\
&+w_0^2 [(w_0-\g)f(w_0-\g)+h_\alpha(w_0-\g)-E h_\alpha(Z_\g)].
}
The lemma is proved by replacing $w$ by $W$, $w_0$ by $W_0$, and taking expectations.
\end{proof}

\begin{proof}[Proof of Lemma \ref{l8}]
We only prove for the case $\g>0$. The case $\g<0$ can be proved similarly.
By Proposition \ref{p1},
\be{
P(W>y)\leq \frac{E e^{xW}}{e^{xy}}\leq C \exp(\frac{x^2}{2}+\frac{\g x^3}{6}-xy).
}
Therefore,
\bes{
&E(1+W^3)I(W\geq x-C\alpha)\\
=& (1+y^3) P(W>y)\Big|_{y=x-C\alpha}+\int_{x-C\alpha}^\infty 3 y^3 P(W>y) dy\\
=&O(1) x^3 \exp(-\frac{x^2}{2}+\frac{\g x^3}{6}).
}
\end{proof}


\section*{Acknowledgements}

Fang X. was partially supported by Hong Kong RGC ECS 24301617 and GRF 14302418, a CUHK direct grant and a CUHK start-up grant.
Shao Q. M. was partially supported by Hong Kong RGC GRF 14302515 and 14304917, and a CUHK direct grant.


\end{document}